\documentclass[12pt]{amsart}
\usepackage[utf8]{inputenc}
\usepackage[T1]{fontenc}

\usepackage{amsmath}
\usepackage{amssymb}
\usepackage{amsthm}
\usepackage{graphicx}
\usepackage{color, soul}
\usepackage{centernot}
\usepackage{verbatim}
\usepackage{multirow}
\usepackage{array}   
\newcolumntype{L}{>{$}l<{$}} 
\usepackage{setspace}
\usepackage{mathrsfs}
\usepackage{enumitem}
\usepackage{tikz}
\usepackage{bigints}
\usepackage{bm}
\usepackage{caption}
\usepackage{subcaption}
\usepackage{url}
\usepackage{imakeidx}
\makeindex[columns=2, intoc]
\usepackage{tikz-cd}

\makeatletter
\renewcommand{\@biblabel}[1]{[#1]\hfill}
\makeatother

\numberwithin{equation}{section}
\newtheorem{theorem}[equation]{Theorem}
\newtheorem{corollary}[equation]{Corollary}
\newtheorem{lemma}[equation]{Lemma}
\newtheorem{proposition}[equation]{Proposition}

\theoremstyle{definition}
\newtheorem{definition}[equation]{Definition}
\newtheorem*{notation}{Notation}
\newtheorem{remark}[equation]{Remark}
\newtheorem{example}[equation]{Example}

\usepackage{mathtools}
\mathtoolsset{showonlyrefs}

\newcommand{\annot}[1]{\text{(\small#1)}}

\newcommand{\der}{\text{d}}

\newcommand{\OO}{\mathcal{O}}

\newcommand{\A}{\mathbb{A}}
\newcommand{\CC}{\mathbb{C}}

\newcommand{\QQ}{\mathbb{Q}}

\newcommand{\MM}{\mathcal{M}}

\newcommand{\Proj}{\mathbb{P}}

\DeclareMathOperator{\ord}{ord}

\DeclareMathOperator{\im}{Im}

\DeclareMathOperator{\Gal}{Gal}

\DeclareMathOperator{\charac}{char}

\DeclareMathOperator{\Hom}{Hom}

\DeclareMathOperator{\Ext}{Ext}
\DeclareMathOperator{\EExt}{\mathbb{E}xt}
\DeclareMathOperator{\Spec}{Spec}
\DeclareMathOperator{\Spf}{Spf}
\DeclareMathOperator{\tame}{tame}
\DeclareMathOperator{\wild}{wild}

\DeclareMathOperator{\Gr}{Gr}

\DeclareFontFamily{U}{wncy}{}
\DeclareFontShape{U}{wncy}{m}{n}{<->wncyr10}{}
\DeclareSymbolFont{mcy}{U}{wncy}{m}{n}
\DeclareMathSymbol{\Sha}{\mathord}{mcy}{"58}

\renewcommand{\bar}{\overline}
\renewcommand{\tilde}{\widetilde}

\usepackage{geometry}
\usepackage{hyperref}

\title[Simply branched covers of curves]{Simply branched covers of curves and wild conductor exponents}
\author{Harry Spencer}

\address{University College, London, WC1H 0AY, UK}
\email{harry.spencer.22@ucl.ac.uk}

\begin{document}

\begin{abstract}
    We use deformation theory to show that a cover of smooth, projective, geometrically connected curves $\pi\colon C\to D$ over a $p$-adic field can be $p$-adically perturbed to obtain a nearby simply branched cover $\pi'\colon C'\to D$ and that, if $D=\Proj^1$ and $\pi^*\OO_{\Proj^1}(1)$ is very ample, then we may take $C'=C$. As an application, we give a formula for the wild conductor exponent of a curve at a prime $p>d$ in terms of the ramification data of any degree $d$ cover $C\to\Proj^1$.
\end{abstract}

\maketitle

\section{Introduction}

A cover of smooth, projective, geometrically connected curves $\pi\colon C\to D$ of degree $d$ over a field $k$ is \textit{simply branched}\footnote{Such a $\pi$ is sometimes said to be \textit{simple}, as in \cite{Fulton}. We avoid this terminology because it may alternatively mean that $k(C)/k(D)$ has no non-trivial intermediate field, as in \cite{RoeXarles}.} if each $Q\in D(\bar{k})$ is such that $\pi^{-1}(Q)$ contains at least $d-1$ geometric points. Equivalently, every geometric ramification point has ramification index $2$, and no fibre contains more than one ramification point.

Over $\CC$, Hurwitz \cite{Hurwitz} systematised the study of simply branched covers of the projective line by endowing the set of such covers of fixed degree from fixed genus curves with the structure of a complex manifold---the \textit{Hurwitz space}. Using the fact that every smooth, projective curve over $\CC$ admits simply branched covers of $\Proj^1$, Severi \cite{Severi} deduced the irreducibility of the moduli space $\MM_g$ of smooth genus $g$ curves over $\CC$.

Fulton's seminal work \cite{Fulton} modernised the study of Hurwitz spaces, constructing schemes $\mathcal{H}_{g,d}$ which parameterise simply branched covers of $\Proj^1$ of degree $d$ from genus $g$ curves over a fixed base scheme---a finite flat cover of curves over a scheme $S$ being simply branched if it is so on each geometric fibre. Over $\Spec k$ for $k$ algebraically closed, Fulton proved that $\mathcal{H}_{g,d}$ is connected for $d>2$ provided $\charac k = 0$ or $\charac k>d$ and deduced the irreducibility of $\MM_g$ in characteristic $p>g+1$. To do so, he showed that genus $g$ curves over algebraically closed fields admit simply branched covers of $\Proj^1$ of all degrees at least $g+1$ (\cite[Proposition 8.1]{Fulton}).

It is a standard principle that simple branching is generic; over $\CC$, an arbitrary cover of curves may be perturbed to a simply branched one. The aim of this article is to establish an arithmetic version of this fact over $p$-adic fields and to apply this to the study of conductors: recall that wild conductor exponents are arithmetic invariants appearing in the conjectural functional equation of a curve's $L$-function (see \cite[\S1 and \S3.1]{Spencer} for definitions and a review of the literature on the topic).

Recent work of the author gives a formula via which wild conductor exponents of curves can be read off from the branch locus of a simply branched cover of $\Proj^1$ (\cite[Theorem 1.1]{Spencer}), renewing arithmetic interest in the study of simple branching. Therein, the author applies \textit{ad hoc} $p$-adic perturbations (\cite[Lemma 6.4]{Spencer}\footnote{See relevant errata \href{https://hspen99.github.io/WildConductors_errata.pdf}{here}.}) to take a coordinate projection map $y\colon C\to\Proj^1$ on certain plane curves and obtain a $p$-adically close curve $\tilde{C}$ with coordinate projection a simply branched cover of $\Proj^1$. Using local constancy results, wild conductor exponents of such curves are thereby determined (\cite[Theorem 1.4]{Spencer}).

\subsection{Deformation-theoretic results}

Herein, we replace these explicit perturbations with deformation-theoretic techniques to generalise this approach to arbitrary covers of curves $C\to D$. In the case that $D=\Proj^1$, we will then use this to remove the simple branching hypothesis from the results of \cite{Spencer}. Some ingredients of the proof may be familiar to experts, but the extension to non-algebraically closed fields and use of $p$-adic approximation seem to be new to the literature. The following theorem provides the $p$-adic approximations required for conductor applications.

\begin{theorem}[{= Theorem \ref{thm:main}}]\label{thm:intro_main}
	Let $K/\QQ_p$ be a finite extension and consider a degree $d$ cover $\pi\colon C\to D$ of smooth, projective, geometrically connected curves over $K$. There exists a smooth, geometrically integral $K$-scheme $T$ with
	\begin{enumerate}
		\item a marked point $0\in T(K)$;
		\item a smooth proper morphism $q_T \colon \mathcal{C} \to T$ with geometrically connected fibres;
		\item a finite flat $T$-morphism $\Pi_T\colon \mathcal{C} \to D\times_K T$ of degree $d$ such that the fibre above $0$ is isomorphic to $\pi\colon C\to D$.
	\end{enumerate}
	Moreover, there is a Zariski-open subscheme $T'$ of $T$ consisting of simply branched covers such that every $p$-adic neighbourhood of $0\in T(K)$ contains a point of $T'(K)$.
\end{theorem}

The proof strategy is inspired by \cite{Schroer}, although we extend those techniques to work over non-algebraically closed fields. First, we lift $\pi$ to a cover $\Pi_1\colon \mathcal{C}_1\to D\times_{K} K[t]/(t^2)$ in a way that separates the ramification points, before lifting iteratively through $K[t]/(t^{n+1})$ and applying Grothendieck algebraisation \cite[Th{\'e}or{\`e}me 5.4.1, Th{\'e}or{\`e}me 5.4.5]{EGA_III} to obtain a morphism $\tilde{\Pi}\colon \mathcal{X}\to D\times_K{K[[t]]}$ with special fibre $\pi$ and simply branched generic fibre. To rigorously ``plug in'' values of $t$ with high valuation, we descend to a finite type family of curves and apply an approximation theorem due to Greenberg \cite[Corollary 1]{Greenberg}. We note that an alternative proof strategy using the theory of stacks is likely viable, but we do not pursue this here. 

A natural question that arises from Theorem \ref{thm:intro_main} is: given $\pi\colon C\to \Proj^1_K$, can we perturb $\pi$ to obtain a simply branched cover $\tilde{\pi}\colon C\to \Proj^1_K$? We also give the following sufficient condition for a positive answer to this question.

\begin{proposition}[{= Proposition \ref{prop:fixed_source}}]
    Let $C$ be a smooth, projective, geometrically connected curve of positive genus over a finite extension $K/\QQ_p$, and let $\pi\colon C\to\Proj^1_K$ be a degree $d$ cover of curves. Write $L=\pi^*\OO_{\Proj^1}(1)$ and let
    \[
        V_0=\pi^*H^0(\Proj^1_K,\OO_{\Proj^1}(1))
        \subset H^0(C,L)
    \]
    be the base-point free pencil defining $\pi$. If $L$ is very ample, then every $p$-adic neighbourhood of $V_0$ in $\Gr\bigl(2,H^0(C,L)\bigr)(K)$ contains a base-point free pencil $V$ such that the induced cover $\tilde\pi\colon C\to\Proj(V^\vee)\cong \Proj^1_K$ is simply branched of degree $d$.
\end{proposition}

\subsection{Application to conductors}

We now explain how Theorem \ref{thm:intro_main} can be used to remove the simple branching hypothesis from the results of \cite{Spencer}.

\begin{notation}
    Given a cover $\pi\colon C\to \Proj^1_K$ of curves over a finite extension $K/\QQ_p$, write 
    \[ w_K(\pi)=\sum_{b\in B}\sum_{r\in \pi^{-1}(b)} \bigl(e_r-1\bigr)\cdot \bigl(v_K(\Delta_{K(b)/K}) -[K(b):K] +f_{K(b)/K}\bigr), \]
    where $B$ is a set of representatives for $\Gal(\bar{K}/K)$-orbits of branch points of $\pi$, $e_r$ is the ramification index of a point $r$ above $b\in B$, and $v_K(\Delta_{K(b)/K})$, $f_{K(b)/K}$ are the normalised valuation of the discriminant and residue degree of $K(b)/K$, respectively. 
\end{notation}

\begin{corollary}\label{cor:conductors}
	Let $K/\QQ_p$ be a finite extension and consider a smooth, projective, geometrically connected curve $C/K$. Fix a cover $\pi\colon C\to \Proj^1_K$ over $K$ such that $\infty\in\Proj^1_K$ is not a branch point. If $p>\deg \pi$, then the wild conductor exponent of $C$ is 
    \[n_{\wild}(C)=w_K(\pi).\]
\end{corollary}

\begin{proof}
    If $\pi$ is simply branched, then \cite[Theorem 1.1]{Spencer} implies that $C$ and a hyperelliptic curve branched over the same locus have the same wild conductor exponents. The conclusion then follows from \cite[Theorem 11.3]{M2D2}.
    
    To reduce to this case, apply Theorem \ref{thm:intro_main} to obtain $q_T\colon \mathcal{C}\to T$ and $\Pi_T\colon \mathcal{C} \to \Proj^1_K\times_K T$ as in the statement of that result. Consider the weighted branch divisor of $\Pi_T$, whose fibres are
    \[\sum_{s\in\Proj^1_{\bar{K}}}\sum_{Q\in \Pi_{T,t}^{-1}(s)}(e_Q-1)\cdot[s] =: B_t.\] 
    This divisor is the pushforward along $\Pi_T$ of the degree $N=2g(C)-2+2\deg\pi$ relative ramification divisor cut out by $\der\Pi_T$. After shrinking $T$ around $0$, it is a relative effective divisor on $\Proj^1_K\times_K T$, corresponding to a morphism $T\to\operatorname{Sym}^N(\Proj^1_K)\cong \Proj^N_K$, $t\mapsto B_t$, which is $p$-adically continuous on $K$-points. Because $\infty$ is not a branch point of $\pi$, we can therefore shrink $T$ to a neighbourhood of $0$ for which $\infty$ is unbranched for all $t$ therein.

    By \cite[Theorem 5.1(1)]{Kisin}, wild conductor exponents are locally constant in the smooth proper family $q_T\colon \mathcal{C} \to T$. That is to say, there is a $p$-adic neighbourhood of $0\in T(K)$ on which the wild conductor exponent is constant on fibres of $q_T$. Thereby fix a neighbourhood $U$ of $0\in T(K)$ on which $n_{\wild}(\mathcal{C}_t)=n_{\wild}(C)$ for all $t\in U$.

   By shrinking $U$ as necessary, we can ensure that $w_K(\Pi_{T,t})=w_K(\pi)$ for all $t\in U\cap T'(K)$ by \cite[Lemma 1.11]{Spencer}, using that $p>\deg\pi$. Choosing $t\in U\cap T'(K)$ by Theorem \ref{thm:intro_main}, we have
    \[
    \renewcommand{\arraystretch}{1.5}
    \begin{array}{r@{\;}c@{\;}l@{\qquad\qquad\qquad}l}
        n_{\wild}(C) &=& n_{\wild}(\mathcal{C}_t)
        & \annot{Constancy of $n_{\wild}$ for $t\in U$} \\
    &=& w_K(\Pi_{T,t})
        & \annot{Simply branched case} \\
        & = & w_K(\pi)
        & \annot{Constancy of $w_K$ for $t\in U\cap T'(K)$},
    \end{array}
    \]
    completing the proof.
\end{proof}

\begin{remark}
    That $\infty$ is not a branch point is not a restriction on the cover $\pi$; this can always be arranged by composing $\pi$ with a M{\"o}bius transformation if necessary.
\end{remark}

\begin{remark}
    \textit{A priori}, wild conductor exponents are easy to compute for $p>2g(C)+1$ (when they vanish), but are otherwise difficult to provably determine. Corollary \ref{cor:conductors} gives a simple method to compute wild conductor exponents of curves over number fields $k$ at primes with residue characteristic greater than the rational gonality, $\gamma_{\text{rat}}$: the minimal degree of a cover $C\to\Proj^1_k$ defined over the ground field. Moreover, by \cite[Lemma 3.4]{Spencer}, we can recover wild conductor exponents of curves at primes $p$ after applying Corollary \ref{cor:conductors} over a tame extension, so we also consider the corresponding tame analogue of rational gonality, $\gamma_{\tame}$.

    For a $p$-adic field $K$ with $p\ge g(C)+1$, we can pass to a tame extension over which $C$ has index $1$ provided $g(C)\ge 2$. By Riemann--Roch, if $d>g(C)$ is divisible by the index, then $C$ admits a map to $\Proj^1_K$ of degree at most $d$, so $\gamma_{\tame} \le g+1$. Moreover, if $C$ admits a unique degree $\gamma$ map to $\Proj^1_{\bar{k}}$ (for example, when $(\gamma-1)^2<g(C)$ and the map is simple by \cite[Theorem 2]{RoeXarles}), then \cite[Theorem 1]{RoeXarles} gives $\gamma=\gamma_{\text{rat}}$ if $C$ has an odd-degree $k$-rational divisor. This uniqueness hypothesis is actually generic: by \cite[Proposition 2.4 and Theorem 2.6]{ArbarelloCornalba}, the generic curve of genus $g$ and gonality $d$ has a unique map to $\Proj^1_{\bar{k}}$ if $2\leq d<\lfloor (g+3)/2\rfloor$. 

    Corollary \ref{cor:conductors} can then be applied to compute conductor exponents for positive genus curves over a number field at primes $p>\gamma$ provided that there is a unique degree $\gamma$ cover of $\Proj^1_{\bar{k}}$. Indeed, \cite[Theorem 1]{RoeXarles} also shows that the unique gonal map descends to a map of degree $\gamma$ from $C$ to a curve of genus $0$ over $k$, which is isomorphic to $\Proj^1$ after a quadratic, hence tame (as $p>\gamma\ge2$), extension.
\end{remark}

\begin{example}
    Consider the genus $12$ complete intersection curve $C$ in $\Proj^1_{\QQ_5}\times \Proj^2_{\QQ_5}$, given in terms of the coordinates $[S:T]$ on $\Proj^1$ and $[U:V:W]$ on $\Proj^2$ by
    \[
    \left\{\begin{array}{l}
    U^2-2V^2-5W^2=0,\\[2mm]
    S^5W^2+(T^5-5S^5)
    \bigl(U^2+UV+2UW+3V^2+4VW+6W^2\bigr)=0.
    \end{array}
    \right.
    \]
    The cover $\pi\colon C\to \Proj^1_{\QQ_5}$ defined by projection onto the first factor,
    \[
    ([S:T],[U:V:W])\mapsto [S:T],
    \]
    has degree $4$. Using Magma \cite{Magma}, we find that---in the affine coordinate $x=T/S$---the geometric branch points with multiplicity
    $\sum_{r\in \pi^{-1}(b)} \bigl(e_r-1\bigr)$ are the roots of
    \[
    \begin{aligned}
        {}&
        (x^5-5)^2\cdot
        \bigl(
        4425167(x^5-5)^4
        +634792(x^5-5)^3
        +90400(x^5-5)^2\\
        &\qquad
        +25152(x^5-5)+1472\bigr).
    \end{aligned}
    \]
    Over $\QQ_5$, the degree $20$ factor splits as two degree $10$ factors and Corollary \ref{cor:conductors} gives: 
    \[
    n_{\wild}(C)
    =
    2\cdot(9-5+1) + (15-10+1)+(10-10+2)
    =
    18.
    \]
    Note that $\pi$ is non-simply branched above the roots of $x^5-5$. We now explain why $\pi$ is the gonal map of $C$, which shows that we cannot determine $n_{\wild}(C)$ directly from the results of \cite{Spencer}.

    After base change to $\overline{\QQ}_5$, the conic $U^2-2V^2-5W^2=0$ is isomorphic to $\Proj^1_{\overline{\QQ}_5}$ and $C_{\overline{\QQ}_5}$ has bidegree $(5,4)$ on $\Proj^1\times\Proj^1$, so projection onto the second factor has degree $5$. If $C_{\overline{\QQ}_5}$ were to admit a map to $\Proj^1_{\overline{\QQ}_5}$ of degree $d\le3$, then this map and the degree $5$ projection would be independent because $\gcd(d,5)=1$. The Castelnuovo--Severi inequality (for example, \cite[Theorem 3.11.3]{Stichtenoth}) would then give $12\leq 4(d-1)$, a contradiction. Therefore, $C$ has gonality $4$. 

    If $\pi$ were to factor through a degree $2$ cover $\pi'\colon C_{\overline{\QQ}_5}\to D$, then $\pi'$ and the degree $5$ projection are independent, so the Castelnuovo--Severi inequality gives $12\leq 2g(D)+4$, and hence $g(D)\geq4$. By Riemann--Hurwitz, $D\to\Proj^1$ therefore has at least $2g(D)+2\geq10$ branch points, each of which has multiplicity at least $2$ in the branch divisor of $\pi$, whereas the displayed polynomial has only five such roots. This is a contradiction, so $\pi$ is simple. Further, because $(4-1)^2<12$, $\pi$ is the unique gonal map of $C$ by \cite[Theorem 2]{RoeXarles}.
\end{example}

\subsection*{Acknowledgements \& Declarations}
I thank Vladimir Dokchitser for his supervision and Andrew Obus for reading a draft of this work. 

OpenAI's ChatGPT suggested approaches to some technical lemmata herein, carried out literature searches and was used for some parts of the typesetting and proofreading process. However, all mathematical claims are verified by the author and this work is entirely written by the author.

This work was supported by the Engineering and Physical Sciences Research Council [EP/S021590/1], the EPSRC Centre for Doctoral Training in Geometry and Number Theory (The London School of Geometry and Number Theory), University College, London.

\section{Deforming covers}

Plainly, the aim of this section is to perturb a cover $\pi\colon C \to D$ to obtain a ``nearby'' simply branched cover $\tilde{\pi}\colon \tilde{C} \to D$. To do so, we will use deformation theory, combining ideas from \cite{Schroer} with an approximation theorem due to Greenberg \cite[Corollary 1]{Greenberg}.

\subsection{Deformation theory}

We begin by recalling some facts about deformations of covers of curves.

\begin{definition}
    Given a finite flat, generically {\'e}tale cover $\pi\colon C\to D$ of curves, over a field $k$ and a local Artinian $k$-algebra $A$ with fixed residue map $A\to k$, a \textit{deformation} of $\pi$ over $A$ is a pair $(\mathcal{C},f)$ with $\mathcal{C}$ a curve over $A$ equipped with a finite flat morphism $f\colon \mathcal{C}\to D\times_k A$ of degree $d$ such that $f \times_A k \cong \pi$.
\end{definition}

To prove Theorem \ref{thm:intro_main}, we need to implement a strategy of Schr{\"o}er used in \cite{Schroer} to take a cover $\pi\colon C\to D$, obtain a deformation over $k[t]/(t^2)$ and successively lift to construct a formal deformation over $k[[t]]$. To do so, we will need the following technical results.

\begin{notation}
    For a cover of curves $\pi\colon C\to D$, we write $L^\bullet_{C/D}$ for the cotangent complex of $\pi$, in the sense of Illusie \cite[Chapter II, Section 1]{Illusie}. By $\EExt^i(L^\bullet_{C/D},\OO_C)$ we denote the $i$-th hyper-Ext group.
\end{notation}

\begin{proposition}\label{prop:eext1}
    Let $k$ be a field, $A$ a local Artinian $k$-algebra, $I\triangleleft A$ an ideal satisfying $\mathfrak{m}_A\cdot I=0$ and $\pi\colon C\to D$ a finite flat, generically {\'e}tale cover of curves over $k$.
    For a deformation $(\mathcal{C},f)$ of $\pi$ over $A/I$, if $\EExt^2(L^\bullet_{C/D},\OO_C)=0$, then the set of isomorphism classes of deformations $(\mathcal{C}',f')$ of $\pi$ over $A$ endowed with an isomorphism $f'\otimes_{A}A/I\cong f$ is a torsor under $\EExt^1(L^\bullet_{C/D},\OO_C)\otimes_k I$.
\end{proposition}

\begin{proof}
    It is shown in \cite[Chapter III, Proposition 2.1.2.3]{Illusie} that the vanishing of a canonical class in $\EExt^2(L^\bullet_{C/D},\OO_C)\otimes_k I$ determines whether $f$ can be extended to such an $f'$, without assuming that $f'$ is finite flat. Finiteness of $f'$ follows from $f$ being finite, whilst \cite[Chapter III, Corollaire 2.1.3.3]{Illusie} ensures flatness. When this canonical class vanishes, which is automatically the case when $\EExt^2(L^\bullet_{C/D},\OO_C)=0$, that same proposition shows the claimed conclusion.
\end{proof}

Illusie also gives a necessary and sufficient condition, which always holds for covers of smooth curves, for the existence of lifted deformations as in Proposition \ref{prop:eext1}.

\begin{proposition}\label{prop:def_lift}
	Let $k$ be a field and $A$ a local Artinian $k$-algebra. Suppose that $I\triangleleft A$ is an ideal satisfying $\mathfrak{m}_A\cdot I=0$ and that $f\colon \mathcal{C}\to D\otimes A/I$ is a deformation over $A/I$ of a cover of smooth curves $\pi\colon C\to D$ over $k$. The deformation $(\mathcal{C},f)$ can be lifted to a deformation $f'\colon \mathcal{C}'\to D\times_k A$ satisfying $f'\times_A A/I\cong f$.
\end{proposition}

\begin{proof}
    By Proposition \ref{prop:eext1}, it suffices to show that
    $\EExt^2(L^\bullet_{C/D},\OO_C)=0$. Under the assumption that
    $\pi$ is a locally complete intersection morphism, this is shown in
    \cite[Proposition 2.1]{Schroer}. A finite type morphism of regular
    Noetherian schemes is locally complete intersection
    \cite[Chapter 6, Example 3.18]{Liu}, hence a cover of smooth curves
    over a field is locally complete intersection.
\end{proof}

Working over an algebraically closed field, Schr{\"o}er showed that first-order deformations are determined by local first-order deformations around branch points:

\begin{lemma}\label{lem:Ext_decomp}
    Let $\pi\colon C\to D$ be a cover of smooth curves over an algebraically closed field $k$ of characteristic $0$ with ramification points $a_i$, writing $b_i=\pi(a_i)$. We have an isomorphism
    \begin{equation}
         \EExt^1(L^\bullet_{C/D}, \OO_{C})  \cong \bigoplus_i \Ext^1\bigl(
        \Omega^1_{\widehat{\OO}_{C,a_i}/
        \widehat{\OO}_{D,b_i}},
        \widehat{\OO}_{C,a_i}
    \bigr).\end{equation} 
\end{lemma}

\begin{proof}
    See \cite[Proposition 2.2]{Schroer}. Note that we have
    \[ \EExt^1(L^\bullet_{C/D},\OO_C) \cong \Ext^1(\Omega^1_{C/D},\OO_C). \]
    Because $\Omega^1_{C/D}$ is supported at the ramification points $a_i$, this Ext group decomposes as the direct sum of completed local contributions:
    \[ \Ext^1(\Omega^1_{C/D},\OO_C) \cong \bigoplus_i \Ext^1\bigl(\Omega^1_{\widehat{\OO}_{C,a_i}/\widehat{\OO}_{D,b_i}},\widehat{\OO}_{C,a_i}\bigr). \]
    Under this identification, the class of a first-order deformation of $\pi$ restricts, in each summand, to the class of the induced first-order deformation of the completed germ.
\end{proof}

Next, we describe the local summands of Lemma \ref{lem:Ext_decomp} in terms of explicit germ deformations:

\begin{lemma}\label{lem:ext_description}
    Let $k$ be an algebraically closed field of characteristic $0$, and let $\pi\colon C\to D$ be a cover of smooth curves over $k$.  Let $a\in C(k)$ be a ramification point of index $e$, and put $b=\pi(a)$. Given uniformisers $x\in \widehat{\OO}_{C,a}$ and $u\in \widehat{\OO}_{D,b}$,
    identify their completed local rings with
    \[
        \widehat{\OO}_{C,a}\simeq k[[x]]
        \qquad \text{and} \qquad
        \widehat{\OO}_{D,b}\simeq k[[u]],
    \]
    respectively. After suitable choices of $u,x$, the completed local map induced by $\pi$ has the form
    \[
        k[[u]]\to k[[x]],
        \qquad
        u\mapsto x^e.
    \]
    We then have
    \[
        \Ext^1\bigl(\Omega^1_{\widehat{\OO}_{C,a}/\widehat{\OO}_{D,b}},\widehat{\OO}_{C,a}\bigr) \cong \Ext^1
        \left(
            \Omega^1_{k[[x]]/k[[u]]},
            k[[x]]
        \right)
        \cong
        k[[x]]/(x^{e-1}).
    \]
    Moreover, under this identification, the first-order deformation of the
    completed local map given by
    \[
        u=x^e+t\cdot g(x),
        \qquad
        t^2=0,
    \]
    corresponds to the class $[g(x)]\in k[[x]]/(x^{e-1})$.
\end{lemma}

\begin{proof}
    Smoothness guarantees that the completed local rings at $a$ and $b$ are complete discrete
    valuation rings with residue field $k$,  hence isomorphic to formal power-series rings
    \[
        \widehat{\OO}_{C,a}\simeq k[[x]],
        \qquad
        \widehat{\OO}_{D,b}\simeq k[[u]],
    \]
    where \(x\) and \(u\) are uniformisers (see, for example, \cite[Chapter 4, Proposition 2.27]{Liu}). Moreover, the ramification index \(e\) means that the pullback of $u$ vanishes to order $e$ at $a$. Because $k$ is algebraically closed and $\charac{k}=0$, the extension of complete discrete valuation rings is tamely ramified and we may choose $u$, $x$, such that the completed local map is given by
    \[
        k[[u]]\to k[[x]],\qquad u\mapsto x^e.
    \]
    For this map, $\Omega^1_{k[[x]]/k[[u]]}=k[[x]]\,dx/(d(x^e))$. Because $d(x^e)=e x^{e-1}dx$, we have
    \[
        \Omega^1_{k[[x]]/k[[u]]}
        \simeq
        k[[x]]/(x^{e-1})\,dx.
    \]

    Consider the short exact sequence
    \[
        0\to k[[x]]
        \xrightarrow{x^{e-1}}
        k[[x]]
        \to
        k[[x]]/(x^{e-1})
        \to 0.
    \]
    Applying $\Hom_{k[[x]]}(-,k[[x]])$ and using that $k[[x]]$ is free over itself, we identify the first Ext group with the cokernel of multiplication by $x^{e-1}$. Hence
    \[
        \Ext^1_{k[[x]]}
        \bigl(
            \Omega^1_{k[[x]]/k[[u]]},
            k[[x]]
        \bigr)
        \cong
        k[[x]]/(x^{e-1}).
    \]

    Finally, we identify the first-order deformation class. A first-order deformation
    \[
        u=x^e+t\cdot g(x),\qquad t^2=0,
    \]
    is obtained from the trivial deformation by perturbing the defining equation of the completed local map. Conversely, after choosing a lift of the parameter $x$, every first-order deformation of the completed germ has this form: as the completed smooth germ $k[[x]]$ is formally smooth over $k$, the source is identified with $\bigl(k[t]/(t^2)\bigr)[[x]]$, and the image of $u$ is a lift of $x^e$, hence is uniquely of the form $x^e+t\cdot g(x)$. The assignment $g\mapsto(u=x^e+t\cdot g(x))$ is additive in $g$ and sends $g=0$ to the trivial deformation. By the standard cotangent-complex description of first-order deformations, compatible with the torsor structure of Proposition \ref{prop:eext1}, this identifies the deformation class with the class of $g(x)$ in the cokernel of multiplication by $\partial_x(x^e-u)=e x^{e-1}$, namely with $[g(x)]\in k[[x]]/(x^{e-1})$. Consistently, an infinitesimal coordinate change $x\mapsto x+t\cdot h(x)$ changes $x^e$ to $x^e+t\cdot e x^{e-1}h(x)$, and so changes $g(x)$ precisely by an element of $(x^{e-1})$.
\end{proof}

To work over non-algebraically closed fields, we will need the following compatibility result.

\begin{lemma}\label{lem:ext_bc}
    Let $k$ be a field and let $\pi\colon C\to D$ be a cover of smooth curves over $k$. There is a natural isomorphism
    \[
        \EExt^1
        \bigl(L^\bullet_{C/D},\OO_C\bigr) \otimes_k\bar k
        \cong
        \EExt^1
        \bigl(
            L^\bullet_{C_{\bar k}/D_{\bar{k}}},
            \OO_{C_{\bar k}}
        \bigr).
    \]
\end{lemma}

\begin{proof}
	As in the proof of Proposition \ref{prop:def_lift}, $\pi$ is locally complete intersection and so its cotangent complex $L^\bullet_{C/D}$ is perfect by \cite[\href{https://stacks.math.columbia.edu/tag/08SL}{\texttt{08SL}}]{Stacks}. Because $\bar{k}$ is flat over $k$, cotangent-complex base change gives
    \[
        Lp^*L^\bullet_{C/D}
        \cong
        L^\bullet_{C_{\bar k}/D_{\bar k}},
    \]
    where $p\colon C_{\bar k}\to C$ is the natural projection (see \cite[\href{https://stacks.math.columbia.edu/tag/08QQ}{\texttt{08QQ}}]{Stacks}). Moreover, because the first argument is perfect, derived Hom commutes with this pullback:
    \[
        Lp^*R\mathcal{H}om_C
        \bigl(L^\bullet_{C/D},\OO_C\bigr)
        \cong
        R\mathcal{H}om_{C_{\bar k}}
        \bigl(
            L^\bullet_{C_{\bar k}/D_{\bar k}},
            \OO_{C_{\bar k}}
        \bigr).
    \]
    Flat base change for cohomology then gives
    \[
    \begin{aligned}
        &R\Gamma
        \bigl(
            C,
            R\mathcal{H}om_C
            (L^\bullet_{C/D},\OO_C)
        \bigr)
        \otimes_k \bar k \cong
        R\Gamma
        \bigl(
            C_{\bar k},
            R\mathcal{H}om_{C_{\bar k}}
            (L^\bullet_{C_{\bar k}/D_{\bar k}},
            \OO_{C_{\bar k}})
        \bigr)
    \end{aligned}
    \]
    by \cite[\href{https://stacks.math.columbia.edu/tag/08IB}{\texttt{08IB}}]{Stacks}. Again by flatness, taking $H^1$ gives the claimed isomorphism.
\end{proof}

\subsection{Main theorem}

We now use the general theory set out above to prove the following theorem:

\begin{theorem}\label{thm:main}
	Let $K/\QQ_p$ be a finite extension and consider a degree $d$ cover $\pi\colon C\to D$ of smooth, projective, geometrically connected curves over $K$. There exists a smooth, geometrically integral $K$-scheme $T$ with
	\begin{enumerate}
		\item a marked point $0\in T(K)$;
		\item a smooth proper morphism $q_T \colon \mathcal{C} \to T$ with geometrically connected fibres;
		\item a finite flat $T$-morphism $\Pi_T\colon \mathcal{C} \to D\times_K T$ of degree $d$ such that the fibre above $0$ is isomorphic to $\pi\colon C\to D$.
	\end{enumerate}
	Moreover, there is a Zariski-open subscheme $T'$ of $T$ consisting of simply branched covers such that every $p$-adic neighbourhood of $0\in T(K)$ contains a point of $T'(K)$.
\end{theorem}

Before proving Theorem \ref{thm:main}, we prove some technical results that will allow us to lift $\pi\colon C\to D$ to a simply branched cover of $D_{K(\!(t)\!)}$. To obtain such a lift, we will iteratively use Proposition \ref{prop:def_lift} to obtain a compatible system. The following will allow us to algebraise such a system to a $K[[t]]$-morphism.

\begin{proposition}\label{prop:algebraise}
    Let $\pi\colon C\to D$ be a degree $d$ cover of smooth, projective, geometrically connected curves over a field $k$ and let $(\Pi_n\colon\mathcal{C}_n\to D_{A_n})_n$ be a compatible system of deformations of $\Pi_0 = \pi$, where $A_n=k[t]/(t^{n+1})$. This system algebraises to a finite flat morphism
    \[
    \widetilde\Pi\colon\mathcal{X}\to D_{k[[t]]}
    \]
    of degree $d$ whose reduction modulo $t^{n+1}$ is $\Pi_n$. Moreover, $q\colon\mathcal{X}\to\Spec k[[t]]$ is smooth and proper with geometrically connected fibres.
\end{proposition}

\begin{proof}
    Following the proof of \cite[Proposition 3.1]{Schroer}, this system gives a formal scheme with flat morphism $\mathfrak{X}\to \Spf k[[t]]$ and finite formal morphism $\mathfrak{X}\to D\times_k\Spf k[[t]]$. Via Grothendieck algebraisation \cite[Th{\'e}or{\`e}me 5.4.5]{EGA_III}, $\mathfrak{X}$ is the formal completion of a flat projective $k[[t]]$-scheme $\mathcal{X}$. We obtain the finite morphism $\tilde{\Pi}\colon \mathcal{X}\to D_{k{[[t]]}}$ via \cite[Th{\'e}or{\`e}me 5.4.1]{EGA_III}. That the reduction modulo $t^{n+1}$ is $\Pi_n$ follows also.

    In addition to Schr{\"o}er's claims, we also claim that $q\colon \mathcal{X}\to \Spec k[[t]]$ is smooth and has geometrically connected fibres, and that $\tilde{\Pi}$ is flat of degree $d$. Note that $q$ is smooth because the image of the non-smooth locus is closed, but does not contain the closed point because the special fibre, $C$, is smooth. Therefore, $q$ is smooth and proper, so it is flat with geometrically reduced fibres. The number of geometric connected components in a fibre of $q$ is locally constant by \cite[\href{https://stacks.math.columbia.edu/tag/0E0N}{\texttt{0E0N}}]{Stacks}. The fibres of $q$ are then geometrically connected because $\Spec k[[t]]$ is connected and the special fibre is geometrically connected.
    
    Because $q$ is smooth, $\mathcal{X}$ is regular. We also have that $D_{k[[t]]}$ is regular, so miracle flatness (\cite[\href{https://stacks.math.columbia.edu/tag/00R4}{\texttt{00R4}}]{Stacks}) gives that $\tilde{\Pi}$ is flat. That $\tilde{\Pi}$ has degree $d$ follows from flatness, using that the degree of the special fibre is $d$. 
\end{proof}

We now show that we can pick a suitable first-order deformation.

\begin{lemma}\label{lem:good_direction}
	Let $\pi\colon C\to D$ be a cover of smooth projective curves over a field $k$ of characteristic $0$ with $\bar{k}$-ramification points $a_i$ of index $e_i$, with $b_i=\pi(a_i)$. Write $E=\EExt^1(L^\bullet_{C/D}, \OO_{C})$ and $E_{\bar{k}}=E\otimes_k\bar{k}$. We have an isomorphism
    \begin{equation}\label{eqn:Ext_decomp}\tag{$\dagger$} E_{\bar{k}} \cong \bigoplus_i \Ext^1\bigl(
        \Omega^1_{\widehat{\OO}_{C_{\bar{k}},a_i}/
        \widehat{\OO}_{D_{\bar{k}},b_i}},
        \widehat{\OO}_{C_{\bar{k}},a_i}
    \bigr).\end{equation} 
    Moreover, there exists $U\subset E$ with $U(k)\ne\emptyset$ such that $U\times_k\bar{k}=\tilde{U}$, where $\tilde{U}\subset E_{\bar{k}}$ is defined as follows: write elements of $\Ext^1\bigl(
        \Omega^1_{\widehat{\OO}_{C_{\bar{k}},a_i}/
        \widehat{\OO}_{D_{\bar{k}},b_i}},
        \widehat{\OO}_{C_{\bar{k}},a_i}
    \bigr)$ as $\lambda_i +\mu_i \cdot x + \cdots$ via Lemma \ref{lem:ext_description}, using the same choice of $u$ for ramification points above the same branch point. Set $\tilde{U}$ to be the complement of the hyperplanes
    \[ \mu_i = 0\hspace{0.1cm}\text{ if }\hspace{0.1cm}e_i\ge3,\hspace{0.6cm} \lambda_i = \lambda_j\hspace{0.1cm}\text{ if }\hspace{0.1cm}b_i=b_j\hspace{0.1cm}\text{ and }\hspace{0.1cm}i\ne j\]
    in $E_{\bar{k}}$ under the identification \eqref{eqn:Ext_decomp}.
\end{lemma}

\begin{proof}
	We have $E_{\bar{k}} \cong \EExt^1(L^\bullet_{C_{\bar{k}}/D_{\bar{k}}}, \OO_{C_{\bar{k}}})$ by Lemma \ref{lem:ext_bc} and so Lemma \ref{lem:Ext_decomp} gives the isomorphism \eqref{eqn:Ext_decomp}. Under this identification, we may describe first-order deformations in the stated way by Lemma \ref{lem:ext_description}.

	We show that the conditions defining $\tilde{U}$ are independent of the choices of parameters $u,x$. Indeed, suppose that $u'$ and $x'$ are another pair of parameters for a ramification point $a_i$ with $u'=x'^{e_i}$. Because these are also uniformisers, we have
    \[ u'=r u+O(u^2), \qquad x'=s x+O(x^2),\]
    with $r=s^{e_i}$. The first-order deformation corresponding to $[\lambda_i +\mu_i \cdot x + \cdots]$ in the original coordinates is then given by
    \[ [r\lambda_i +  r\mu_i\cdot x + \cdots] =
    [r\lambda_i+(r/s)\mu_ix'+\cdots]
    \]
    in the new pair. In particular, the condition $\mu_i\ne 0$ is independent of the choice of parameters. Moreover, if $b_i=b_j$, then the same change of target parameter $u'=r u+O(u^2)$ is used at both points. Hence the condition $\lambda_i\ne\lambda_j$ is also independent of any choices.

    Now, clearly, $\tilde{U}$ is a non-empty Zariski-open subset of $E_{\bar{k}}$ (viewed as an affine space over $\Spec \bar{k}$). Also, $\Gal(\bar{k}/k)$ preserves ramification indices and the property of lying in the same fibre, and so permutes the defined hyperplanes. Therefore, the closed complement is Galois-invariant and faithfully flat descent gives closed $V\subset E$ such that $U = E\setminus V$ has $U\times_k \bar{k}=\tilde{U}$. $E$ is affine over $k$ and $U$ is a non-empty Zariski-open subset, so $U(k)\ne \emptyset$ because $k$ is infinite and $E$ is finite-dimensional.
\end{proof}

The next lemma will show that a choice of first-order deformation corresponding to an element of $U(k)$ as above will be sufficient to guarantee that lifting to $k[[t]]$ will yield a simply branched generic fibre.

\begin{lemma}\label{lem:local_splitting}
    Let $k$ be an algebraically closed field of characteristic $0$, and take $e\ge 2$. Consider a formal deformation of the ramification germ $u=x^e$ of the form
    \[
        u= x^e+t(\lambda+\mu x+a_2x^2+\cdots)+t^2H(x,t)=:F(x,t),
    \]
    with $H(x,t)\in k[[x,t]]$. If $e\ge 3$, then assume $\mu\neq0$. On the geometric generic fibre over $\bar{k(\!(t)\!)}$, the ramification specialising to $x=0$ consists of $e-1$ distinct ramification points of ramification index $2$. Moreover, the corresponding branch values are pairwise-distinct with common $t$-coefficient $\lambda$.
\end{lemma}

\begin{proof}
    The ramification points are the solutions to
    \begin{equation}\label{eqn:crit}\tag{$\ddagger$}
        \frac{\partial F}{\partial x}
        =
        e x^{e-1}
        +t(\mu+2a_2x+3a_3x^2+\cdots)
        +t^2\frac{\partial H}{\partial x}
        =
        0.
    \end{equation}
    Modulo $t$, the central expression of \eqref{eqn:crit} is $e x^{e-1}$. Weierstrass preparation over the complete local ring $k[[t]]$ (for example, \cite[Theorem IV.9.2]{Lang}) allows us to write that expression as a unit in $k[[t]][[x]]$ multiplied by a monic polynomial $P(x)\in k[[t]][x]$ of degree $e-1$, whose non-leading coefficients lie in $(t)$. Thus the solutions of \eqref{eqn:crit} specialising to $x=0$ are precisely the roots of $P$, counted with multiplicity.

    First suppose $e\ge3$.  The Newton polygon of this equation has a single lower segment joining $(0,1)$ to $(e-1,0)$: the constant term has $t$-adic valuation $1$ (because $\mu\ne0$), the leading coefficient of $P$ is a unit, and all intermediate coefficients have valuation at least $1$. Hence all $e-1$ roots have valuation $1/(e-1)$.

    Setting $x=t^{1/(e-1)}y$, dividing \eqref{eqn:crit} by $t$ and reducing modulo $t^{1/(e-1)}$, we obtain
    \[
        e y^{e-1}+\mu=0,
    \]
    which has $e-1$ distinct roots. By Hensel's lemma, the critical equation has $e-1$ distinct roots
    \[
        x_j=t^{1/(e-1)}y_j+\text{higher order terms}
    \]
    over $\overline{k((t))}$, where the $y_j$ are the distinct solutions of $e y^{e-1}+\mu=0$. We also have
    \[
        \frac{\partial^2F}{\partial x^2}
        =
        e(e-1)x^{e-2}+O(t)
    \]
    and evaluating at $x_j$ gives
    \[
        e(e-1)x_j^{e-2} + O(t)
    \]
    having valuation $(e-2)/(e-1)$. In particular, the second derivative does not vanish at $x_j$ and the corresponding ramification point has index $2$.

    Substituting $x_j=t^{1/(e-1)}y_j+\cdots$ back into $F$ gives
    \[
        F(x_j,t)
        =
        \lambda t
        +
        t^{e/(e-1)}
        \bigl(y_j^e+\mu y_j\bigr)
        +
        \text{higher order terms}
    \]
    and, using $e y_j^{e-1}+\mu=0$, we have
    \[
        y_j^e+\mu y_j
        =
        -(e-1)y_j^e.
    \]
    The values $y_j$ are distinct; $y_{j'}$ and $y_j$ differ by a non-trivial $(e-1)$-th root of unity for $j'\ne j$. Therefore, the values $-(e-1)y_j^e$ are also distinct, which shows that the critical values in this cluster are pairwise-distinct.

    For $e=2$, the critical
    equation is
    \[
        2x+t(\mu+2a_2x+\cdots)+t^2\frac{\partial H}{\partial x}=0,
    \]
    which has second derivative $2+O(t)$ and a unique solution near $x=0$. This shows that the unique nearby ramification point has index $2$. The corresponding critical value is $\lambda t + O(t^2)$.
\end{proof}

We now show that simple branching is an open condition in families of covers.

\begin{lemma}\label{lem:simply_open}
    Let $S$ be a locally Noetherian scheme over a field $k$ of characteristic $0$. Let $q\colon\mathcal{C}\to S$ and $r\colon\mathcal{D}\to S$ be smooth proper relative curves equipped with a finite locally free $S$-morphism $f\colon\mathcal{C}\to \mathcal{D}$. The set
    \[
        S'=\{s\in S\mid f_s\colon\mathcal{C}_s\to\mathcal{D}_s\text{ is simply branched}\}
    \]
    is Zariski-open in $S$.
\end{lemma}

\begin{proof}
    The relative differential gives a morphism of line bundles $df\colon f^*\Omega^1_{\mathcal{D}/S}\to \Omega^1_{\mathcal{C}/S}$. Equivalently, $df$ is a section of the line bundle $\Omega^1_{\mathcal{C}/S}\otimes\bigl(f^*\Omega^1_{\mathcal{D}/S}\bigr)^\vee$. Write $R\subset \mathcal{C}$ for the vanishing of this section. For every geometric point $\bar{s}\to S$, the finite morphism $f_{\bar{s}}$ is generically {\'e}tale, so $df_{\bar{s}}$ does not vanish identically on any irreducible component of $\mathcal{C}_{\bar{s}}$. Thus $R_{\bar{s}}$ is an effective Cartier divisor on $\mathcal{C}_{\bar{s}}$. Because $\mathcal{C}\to S$ is smooth and $R$ is locally principal, the fibrewise criterion for relative effective Cartier divisors \cite[\href{https://stacks.math.columbia.edu/tag/062Y}{\texttt{062Y}}]{Stacks} shows that $R$ is a relative effective Cartier divisor on $\mathcal{C}/S$. In particular, $R\to S$ is flat.

    Each $R_{\bar{s}}$ is finite. Because $R$ is closed in $\mathcal{C}$ and $\mathcal{C}\to S$ is proper, the morphism $R\to S$ is proper with finite fibres, hence finite. Together with the flatness above, this shows that $R\to S$ is finite locally free.

    Write $E\subset R$ for the {\'e}tale locus of $R\to S$, which is open by definition. Because $R\to S$ is finite, the image of the closed subset $R\setminus E$ is closed. Hence
    \[ U= S \setminus \im(R\setminus E\to S) \]
    is open and $R_{U}\to U$ is finite {\'e}tale. Because $R\to S$ is flat and locally of finite presentation, the fibrewise criterion for smoothness \cite[\href{https://stacks.math.columbia.edu/tag/01V9}{\texttt{01V9}}]{Stacks}, applied in relative dimension zero, shows that a geometric point $\bar{s}\to S$ factors through $U$ precisely when $R_{\bar{s}}\to \bar{s}$ is {\'e}tale. At a ramification point of index $e$, the differential $df_{\bar{s}}$ vanishes to order $e-1$. Hence $R_{\bar{s}}\to\bar{s}$ is {\'e}tale precisely when every ramification point of $f_{\bar{s}}$ has ramification index $2$.
    
    It remains to impose the condition that no two ramification points lie in the same branched fibre. To do so, we work in 
    \[R_U\times_{\mathcal{D}_U}R_U=\{(x,y)\in R_U\times_U R_U:\ f(x)=f(y)\}\]
    and consider the space parameterising ordered pairs of distinct ramification points lying in the same fibre of $f$:
    \[ W = \bigl(R_{U}\times_{\mathcal{D}_{U}}R_{U}\bigr) \setminus\Delta_{R_{U}}, \]
    where $\Delta_{R_{U}}$ denotes the diagonal in $R_U\times_U R_U$ restricted to
    $R_{U}\times_{\mathcal{D}_U}R_{U}$. Because $R_{U}\to U$ is finite {\'e}tale, so unramified, the diagonal is open in $R_U\times_{U} R_U$. Intersecting with $R_U\times_{\mathcal{D}_U}R_U$, we see that $\Delta_{R_{U}}$ is open in $R_{U}\times_{\mathcal{D}_U}R_{U}$, and so $W$ is closed in $R_{U}\times_{\mathcal{D}_{U}}R_{U}$. Because $W\to U$ is finite, its image in $U$ is closed. Therefore, the simply branched locus $S'=U\setminus \im(W\to U)$ is open in $S$.
\end{proof}

Finally, we come to the proof of Theorem \ref{thm:main}.

\begin{proof}[Proof of Theorem \ref{thm:main}]
	Fix $\xi\in U(K)\subset \EExt^1(L^\bullet_{C/D},\OO_C)$ as in Lemma \ref{lem:good_direction}. As in Proposition \ref{prop:eext1}, this corresponds to a first-order deformation after setting the origin to be the trivial deformation. By iterative application of Proposition \ref{prop:def_lift}, we lift to a formal deformation over $K[[t]]$: set $A_n = K[t]/(t^{n+1})$ and take $I_n\triangleleft A_{n+1}$ to be $t^{n+1}\cdot K[t]/(t^{n+2})$. This gives a compatible system of deformations $\pi_n\colon \mathcal{C}_n \to D \otimes_K A_n$. By Proposition \ref{prop:algebraise}, we thereby obtain a finite flat morphism $\tilde{\Pi}\colon \mathcal{X}\to D_{K[[t]]}$ of degree $d$ with special fibre $\pi$. Because $q\colon \mathcal{X}\to \Spec K[[t]]$ is smooth and proper with geometrically connected fibres, $C$ and the generic fibre have the same genus.

    The conditions on $\tilde{U}$ as in Lemma \ref{lem:good_direction} ensure that the generic fibre is a simply branched cover of curves over $K(\!(t)\!)$: base changing to $\bar{K}$ and completing at $a_i$, smoothness of $q$
    identifies the completed local map of $\tilde{\Pi}$ with
    $u=F_i(x,t)\in\bar{K}[[x,t]]$, where
    $F_i\equiv x^{e_i}+t(\lambda_i+\mu_ix+\cdots)\pmod{t^2}$ with
    $\lambda_i+\mu_ix+\cdots$ the $i$-th component of $\xi$ under
    \eqref{eqn:Ext_decomp}, by Lemmas \ref{lem:Ext_decomp} and
    \ref{lem:ext_description}. Then, by Lemma \ref{lem:local_splitting}, the condition $\mu_i\ne 0$ ensures that the ramification point $a_i$ splits into $e_i-1$ distinct ramification points of index $2$ with pairwise-distinct branch values in a local cluster. Moreover, each branch value in the cluster attached to $a_i$ has leading term $\lambda_i t$, so the $\lambda_i\ne \lambda_j$ condition ensures that ramification points in the same fibre are separated into different fibres. Using that $C$ and the generic fibre have the same genus, Riemann--Hurwitz guarantees that no additional branching is introduced away from these local clusters: a contribution of $e_i-1$ is replaced by $e_i-1$ separate contributions of $2-1=1$.

	We now descend to a finite type family of covers. Indeed, let $(B_j)_j$ be the direct system of finitely generated $K$-subalgebras of $K[[t]]$, ordered by inclusion. Fixing a sufficiently large index $j_0$ and writing $S=\Spec B_{j_0}$, we have that $\mathcal{X}\to \Spec K[[t]]$ descends to a morphism $q\colon X\to S$ by \cite[\href{https://stacks.math.columbia.edu/tag/01ZM}{\texttt{01ZM}}]{Stacks}, using that the former is finitely presented. Similarly, $\tilde{\Pi}$ descends to a family $\Pi\colon X\to D\times_K S$ with $\Pi$ finite and locally free of degree $d$ (see \cite[\href{https://stacks.math.columbia.edu/tag/06AC}{\texttt{06AC}}]{Stacks}). We may also assume that $q$ is smooth, proper and of relative dimension $1$ by further enlarging $B_{j_0}$ if need be (see \cite[\href{https://stacks.math.columbia.edu/tag/0C0C}{\texttt{0C0C}}, \href{https://stacks.math.columbia.edu/tag/081F}{\texttt{081F}} and \href{https://stacks.math.columbia.edu/tag/0EY2}{\texttt{0EY2}}]{Stacks}).

	Note that $B_{j_0}\hookrightarrow K[[t]] \twoheadrightarrow K$ defines a point $s_0\in S(K)$ such that the fibre over $s_0$ is $\pi\colon C \to D$. Because $C$ is geometrically connected, by again enlarging $B_{j_0}$ if necessary, we may shrink $S$ to a neighbourhood of $s_0$ and thereby further assume that the fibres of $q$ are geometrically connected. 

	Write $S'=\{s\in S \mid \Pi_s\text{ is simply branched}\}$ and note that this is Zariski-open in $S$ by Lemma \ref{lem:simply_open}. Because the generic fibre of $\tilde{\Pi}$ is simply branched, the $K[[t]]$-point $\gamma$ coming from the inclusion $B_{j_0}\hookrightarrow K[[t]]$ is not contained in the Zariski-closed complement $\Delta:=S\setminus S'$. In particular, we can find $h$ vanishing on $\Delta$ such that $\gamma^*h\in K[[t]]$ is non-zero. Write $M=\ord_t(\gamma^*h)$.

	Set $R=K[t]_{(t)}^\mathrm{hens.}$ to be the Henselisation of the localisation of $K[t]$ at $(t)$, so the completion of $R$ is $K[[t]]$. We obtain from $\gamma$ a $K[[t]]$-point of the finite type $R$-scheme $S_R = S\times_K \Spec R$. By Greenberg's approximation theorem \cite[Corollary 1]{Greenberg}, there exists an $R$-point $\gamma_R\colon \Spec R\to S_R$ agreeing with $\gamma$ modulo $t^{M+1}$. Also writing $\gamma_R$ for the induced $K$-morphism $\Spec R\to S_R \to S$, we have that $\gamma_R^*h\ne 0$, so the image under $\gamma_R$ of the generic point of $R$ lies in $S'$.

	$R$ is the filtered colimit of the local rings of pointed {\'e}tale neighbourhoods of $0\in \A^1_K$ (see \cite[Chapter VIII]{Raynaud}). Because $S$, $q$ and $\Pi$ are all finitely presented, $\gamma_R$ descends to a morphism $\gamma_T\colon (T,0) \to (S,s_0)$ from some pointed {\'e}tale neighbourhood $(T,0)\to (\A^1_K,0)$. After restricting to a neighbourhood of $0\in T(K)$, we may assume that $T$ is smooth and geometrically integral. Pulling back $q$ and $\Pi$ under $\gamma_T$, we obtain $q_T$ and $\Pi_T$ as claimed.
	
	Finally, write $T'=\{t\in T \mid \Pi_{T,t}\text{ is simply branched}\}$. Because $\gamma_R^*h\ne 0$, the non-simply branched locus $T\setminus T'$ on $T$ is a proper closed subset. After restricting to a neighbourhood of $0\in T$, we may assume that $T$ is a smooth, geometrically integral curve with $T\to\A^1_K$ étale, and that $T\setminus T'$ is finite. The induced map on $K$-points is a local analytic isomorphism near $0$, and a finite subset of a $p$-adic analytic neighbourhood has empty interior. Therefore, every $p$-adic neighbourhood of $0\in T(K)$ contains a point of $T'(K)$.
\end{proof}

\section{Fixed source perturbations}

In this section, we consider the question: when can we perturb $\pi\colon C\to \Proj^1$ so as to obtain a simply branched cover $\pi'\colon C\to\Proj^1$ from the same curve? In particular, we show that the answer is in the affirmative if $\pi^*\OO_{\Proj^1}(1)$ is very ample.

\begin{lemma}\label{lem:projection}
    Let $k$ be a field of characteristic $0$, and let $C\subset\Proj^r_k$ be a non-degenerate smooth, projective, geometrically connected curve of positive genus. There is a non-empty
    Zariski-open subset
    \[
        \Omega_C\subset\Gr(r-2,\Proj^r_k)
    \]
    such that, for every $\Lambda\in\Omega_C(\bar{k})$, one has $\Lambda\cap C_{\bar{k}}=\emptyset$ and projection from $\Lambda$ induces a simply branched cover $p_\Lambda:C_{\bar{k}}\to\Proj^1_{\bar{k}}$.
\end{lemma}

\begin{proof}
    Let $G^\circ\subset \Gr(r-2,\Proj^r_k)$ be the open subscheme parameterising codimension-two linear subspaces $\Lambda\subset \Proj^r_k$ with $\Lambda\cap C=\emptyset$. Let
    \[
        \mathcal{I}
        =
        \{(\Lambda,H)\in G^\circ\times_k(\Proj^r_k)^\vee : \Lambda\subset H\}.
    \]
    Then $\mathcal{I}\to G^\circ$ is a $\Proj^1$-bundle: its fibre over $\Lambda$ is the pencil of hyperplanes containing $\Lambda$. For every $\Lambda\in G^\circ$, we have $\Lambda\cap C=\emptyset$, so $(P,\Lambda)\in C\times_k G^\circ$ determines a unique hyperplane $H\subset\Proj^r_k$ spanned by $P$ and $\Lambda$. Therefore, $(P,\Lambda)\mapsto (\Lambda,\langle \Lambda,P\rangle)$ defines a morphism
    \[
        p\colon C\times_k G^\circ\to \mathcal{I}.
    \]
    The fibre of $p$ over a point $\Lambda\in G^\circ$ is the usual projection of $C$ from $\Lambda$. Non-degeneracy of $C$ ensures that $p$ is quasi-finite, so finite by properness. Because $C\times_k G^\circ$ and $\mathcal{I}$ are smooth relative curves over $G^\circ$, $p$ is finite locally free. By Lemma \ref{lem:simply_open}, the locus $\Omega_C\subset G^\circ$ parameterising centres $\Lambda$ with projection $p_\Lambda\colon C_{\bar{k}}\to\Proj^1_{\bar{k}}$ simply branched is Zariski-open.

    It remains to show that $\Omega_C(\bar{k})\ne\emptyset$, so we now work over this algebraic closure. The existence theorem for Lefschetz pencils \cite[Expos\'{e} XVII, Th\'{e}or\`{e}me 2.5]{KatzLefschetz} shows that $C\hookrightarrow\Proj^r$ is a Lefschetz embedding, for which we may choose a Lefschetz pencil $M$. That is to say, we obtain a pencil $M$ such that $\Lambda_M=\cap_{H\in M} H$ meets $C$ transversely, the set $\Gamma$ of hyperplanes $H\in M$ meeting $C$ transversely is open and dense in $M$, and every $H\in M\setminus\Gamma$ meets $C$ transversely except at a unique point, where the hyperplane section has an ordinary quadratic singularity. For a curve, the first condition simplifies to $\Lambda_M\cap C=\emptyset$, whilst the last means that any $H\in M\setminus\Gamma$ has exactly one double intersection with $C$. Hence the associated projection
    \[
        p_{\Lambda_M}\colon C_{\bar{k}}\to M\simeq\Proj^1_{\bar{k}},
    \]
    whose fibre above $H\in M$ is $H\cap C_{\bar{k}}$, is simply branched and so $\Lambda_M\in\Omega_C(\bar{k})$.
\end{proof}

\begin{proposition}\label{prop:fixed_source}
    Let $C$ be a smooth, projective, geometrically connected curve of positive genus over a finite extension $K/\QQ_p$, and let $\pi\colon C\to\Proj^1_K$ be a degree $d$ cover of curves. Write $L=\pi^*\OO_{\Proj^1}(1)$ and let
    \[
        V_0=\pi^*H^0(\Proj^1_K,\OO_{\Proj^1}(1))
        \subset H^0(C,L)
    \]
    be the base-point free pencil defining $\pi$. If $L$ is very ample, then every $p$-adic neighbourhood of $V_0$ in $\Gr\bigl(2,H^0(C,L)\bigr)(K)$ contains a base-point free pencil $V$ such that the induced cover $\tilde\pi\colon C\to\Proj(V^\vee)\cong \Proj^1_K$ is simply branched of degree $d$.
\end{proposition}

\begin{proof}
    Write $W=H^0(C,L)$ and $r=\dim_K(W)-1$. By very ampleness, we have a closed immersion $\varphi_L\colon C\hookrightarrow\Proj(W^\vee)$. For a two-dimensional subspace $V\subset W$, define
    \[
        \Lambda_V:=\Proj(\operatorname{Ann}(V))\subset\Proj(W^\vee).
    \]
    Then $V$ is base-point free on $C$ if and only if $\Lambda_V\cap\varphi_L(C)=\emptyset$, in which case the morphism defined by $V$ is the projection of $\varphi_L(C)$ from $\Lambda_V$.

    By Lemma \ref{lem:projection}, a non-empty Zariski-open subset $\Omega$ of codimension-two centres gives finite simply branched projections on $\varphi_L(C)$. Under Grassmannian duality
    \[
        \Gr(2,W)\xrightarrow{\sim}\Gr(r-2,\Proj(W^\vee)),
        \qquad V\mapsto\Lambda_V,
    \]
    its inverse image is a non-empty Zariski-open subset $\Omega'\subset\Gr(2,W)$, in which every pencil is base-point free and defines a simply branched map of degree $d$. Choose an affine chart $A$ containing $V_0$. Because $\Omega'\cap A$ is open in an affine space, defined over $K$ and non-empty over $\bar{K}$, $\Omega'(K)$ meets every $p$-adic neighbourhood of $V_0$.
\end{proof}

\begin{remark}
    In Proposition \ref{prop:fixed_source}, it suffices to assume that $d\ge 2g(C)+1$ because all line bundles of degree at least $2g(C)+1$ on a smooth curve $C$ are very ample \cite[Chapter IV, Corollary 3.2(b)]{Hartshorne}. In this case, however, there is no application to the determination of conductor exponents.
\end{remark}

\begin{example}
    Let $K/\QQ_p$ be a finite extension. For $C\subset \Proj^2_K$ a smooth plane curve of degree at least $3$, projection from a point $P\in \bigl(\Proj^2\setminus C\bigr)(K)$ defines a morphism $\pi_P\colon C\to \Proj^1_K$. Fibres are cut out by lines through $P$, so $\pi_P^*\OO_{\Proj^1}(1)\cong \OO_C(1)$, which is very ample. Consequently, $\pi_P$ can always be perturbed to a nearby simply branched cover.
\end{example}

\bibliographystyle{plain}
\bibliography{main}

\end{document}